\documentclass[12pt,letterpaper]{amsart}

\usepackage{rjmacros}
\newcommand{\wtcR}{\wtilde{\cR}}
\newcommand{\whf}{\what{f}}
\newcommand{\whg}{\what{g}}
\newcommand{\whh}{\what{h}}
\newcommand{\whk}{\what{k}}

\newcommand{\whfN}{\what{\fN}}
\newcommand{\vp}{\varphi}
\newcommand{\wtGm}{\widetilde{\Gamma}}
\begin{document}
\title[Weyl families of essentially unitary pairs]{Weyl
families of essentially unitary pairs}
\author{Rytis Jur\v{s}\.{e}nas}
\address{Institute of Theoretical Physics and Astronomy,
Vilnius University, Sauletekio 3, LT-10257, Vilnius, Lithuania}
\email{rytis.jursenas@tfai.vu.lt}
\date{\today}
\begin{abstract}
It is known that the Weyl families corresponding to
unitary boundary pairs $(\cH,\Gamma)$ belong to the class
$\wtcR(\cH)$ of Nevanlinna families. Here we extend
the theorem to the case of essentially unitary boundary pairs
by showing that the closures of members of the Weyl
families belong to the class $\wtcR(\cH)$. Thus bounded
Weyl functions of essentially unitary boundary pairs are of
class $\cR[\cH]$.
\end{abstract}
\maketitle
\section{Introduction}
Throughout $\fH$ and $\cH$ denote Hilbert spaces.
Let $\Gamma\subseteq\fH^2\times\cH^2$ be a linear relation
from a $J_\fH$-space to a $J_\cH$-space \cite[Section~1]{Azizov89},
where the canonical symmetry $J_\fH$ ($J_\cH$)
acts on $\fH^2$ ($\cH^2$)
as the operator of multiplication by the matrix
$\bigl(\begin{smallmatrix}0 & -\img I_\fH \\ \img I_\fH & 0
\end{smallmatrix}\bigr)$.
Let $\Gamma^{[*]}$ denote the Krein space adjoint of $\Gamma$
\cite[Equation~(2.4)]{Derkach17}, \cite[Section~7.2]{Derkach12}.
Then $\Gamma$ is said to be $(J_\fH,J_\cH)$-isometric
if $\Gamma^{-1}\subseteq\Gamma^{[*]}$ and
$(J_\fH,J_\cH)$-unitary if $\Gamma^{-1}=\Gamma^{[*]}$
\cite[Definition~2.1]{Derkach17},
and essentially $(J_\fH,J_\cH)$-unitary if
$\ol{\Gamma}^{\;-1}=\Gamma^{[*]}$ \cite[Section~2.3]{Derkach06},
where the overbar denotes the closure
(with respect to the Hilbert space topology).
In what follows we refer to $\Gamma$
simply as isometric/(essentially) unitary.

Let $\Gamma$ be a linear relation
from a $J_\fH$-space to a $J_\cH$-space, and put
($[\bot]$ denotes the $J_\fH$-orthogonal complement;
the superscript $*$ denotes the Hilbert space adjoint)
\[
A_*:=\dom\Gamma\,,\quad
A:=A^*_*\equiv(A_*)^*=(\dom\Gamma)^{[\bot]}
=\mul\Gamma^{[*]}\,.
\]
Then $A$ is a closed, not necessarily symmetric,
linear relation in $\fH$, and
$A_*$ is dense in $A^*=\ol{\dom}\Gamma$
(with respect to the Hilbert space topology on $\fH^2$).
If $\Gamma$ is unitary,
then $A=S:=\ker\Gamma$ is automatically symmetric;
if $\Gamma$ is essentially unitary, then
$A=\ker\ol{\Gamma}\subseteq A^*=S^*$ is also
symmetric. Let
\[
A_\star:=\dom\ol{\Gamma}\,.
\]
Then $\ol{A_\star}=\ol{A_*}=A^*$ but
$A_\star\supseteq A_*$.
For a unitary $\Gamma$, $A_\star=A_*$.

Let $A$ be a closed symmetric linear relation in $\fH$.
The pair $(\cH,\Gamma)$ is an isometric (unitary) boundary
pair for $A^*$ if $A_*\subseteq A^*$ densely with
respect to the topology on $\fH^2$, and if
$\Gamma$ is isometric (unitary)
\cite[Definition~3.1]{Derkach17}.
Note that $\ol{A_*}=A^*$ implies that $A=\mul\Gamma^{[*]}$.
If $(\cH,\Gamma)$ is an isometric boundary pair for
$A^*$, then $(\cH,\ol{\Gamma})$ is also an isometric
boundary pair for $A^*$.

Let $(\cH,\Gamma)$ be a unitary boundary pair for $A^*=S^*$.
In the terminology of \cite[Definition~3.1]{Derkach06},
a unitary $\Gamma$ is called a boundary relation
for $S^*$; see also \cite[Proposition~3.2]{Derkach06}.
If $\Gamma$ is essentially unitary, we say that
the pair $(\cH,\Gamma)$ is an essentially unitary boundary pair
for $A^*=S^*$.

The first part of \cite[Theorem~3.9]{Derkach06} states that
the Weyl family $M_\Gamma(z)$, $z\in\bbC_*:=\bbC\setm\bbR$,
corresponding to a boundary relation $\Gamma$
for $S^*$ is a
Nevanlinna family; that is, it belongs to the class $\wtcR(\cH)$
(Definition~\ref{defn:Nevaln}).
Here we prove an analogue of this statement for an
essentially unitary $\Gamma$.
\begin{thm}\label{thm:main}
Let $M_\Gamma$ be the Weyl family
corresponding to an essentially unitary boundary pair $(\cH,\Gamma)$
for $A^*$.
Then the closure $\bbC_*\ni z\mapsto
\ol{M_\Gamma(z)}=M_{\ol{\Gamma}}(z)$
belongs to a Nevanlinna family.
\end{thm}
Here $M_{\ol{\Gamma}}$ is the Weyl family corresponding
to a unitary boundary pair $(\cH,\ol{\Gamma})$.
For $\Gamma$ unitary (hence closed), the theorem clearly
reduces to the first part of \cite[Theorem~3.9]{Derkach06}.
By assuming additionally that $M_\Gamma(z)\in \mrm{B}(\cH)$ is
a bounded (hence closed) operator, one deduces another corollary.
\begin{cor}\label{cor:main}
Let $(\cH,\Gamma)$ be an essentially unitary boundary pair
for $A^*$, and assume in addition that
$M_\Gamma(z)\in\mrm{B}(\cH)$, $z\in\bbC_*$.
Then the Weyl function $M_\Gamma=M_{\ol{\Gamma}}$
belongs to the subclass $\cR[\cH]$ of Nevanlinna functions.\qed
\end{cor}
According to \cite[Proposition~5.9]{Derkach06} a Nevanlinna function
of class $\cR[\cH]$ can be realized as the Weyl function of
a $B$-generalized boundary pair $(\cH,\Gamma)$
\cite[Definition~3.5]{Derkach17}. Let us recall that
ordinary, $B$-generalized, $S$-generalized, $ES$-generalized
boundary pairs are all unitary boundary pairs; see
\cite{Derkach17} for more details. Yet Corollary~\ref{cor:main}
shows that one can find a non-unitary boundary pair
with the same Weyl function.

Assuming the hypotheses in Corollary~\ref{cor:main}
and in addition $\ran\ol{\Gamma}=\cH^2$, one concludes that
the Weyl function $M_\Gamma=M_{\ol{\Gamma}}$
belongs to the subclass $\cR^u[\cH]$ of uniformly strict
Nevanlinna functions. The single-valued linear relation
(\ie operator)
$\Gamma$ with such properties arises, for example,
in the study of triplet extensions of self-adjoint operators
\cite{Jursenas18}; see also example in Section~\ref{sec:exam}.

The proof of Theorem~\ref{thm:main} is organized as follows:
In Section~\ref{sec:iso} we state and prove the main technical
lemma (Lemma~\ref{lem:main}).
In Section~\ref{sec:Weyl} we compute the adjoint
$M_\Gamma(z)^*$ for an isometric boundary pair $(\cH,\Gamma)$;
it follows that $M_\Gamma(z)^*=M_{\ol{\Gamma}}(\ol{z})$
for $\Gamma$ essentially unitary.
Since $M_{\ol{\Gamma}}$ is a Nevanlinna
family for $\ol{\Gamma}$ unitary,
this leads to Theorem~\ref{thm:main}; see Section~\ref{sec:Nevaln}.

Throughout we use the standard symbols $\dom$, $\ran$,
$\mul$, and $\ker$ to denote the domain, the range,
the multivalued part, and the kernel of a linear relation.
For more details related to the theory of linear relations and
Nevanlinna families the reader may consult the papers in
\cite{Boitsev18,Derkach17b,Behrndt15,Behrndt13,Snoo11,Derkach09,
Hassi09, Behrndt08,Hassi07,Hassi96,Derkach91}
and also an extensive list of references therein.
\section{Main lemma}\label{sec:iso}
Consider a linear relation $\Gamma$ from a $J_\fH$-space
to a $J_\cH$-space.
The $J_\fH$-metric
$[\cdot,\cdot]_{\fH^2}$ is written in terms of the
$\fH^2$-scalar product $\braket{\cdot,\cdot}_{\fH^2}$
according to
\[
[\whf,\whg]_{\fH^2}:=\braket{\whf,J_\fH\whg}_{\fH^2}
=-\img(\braket{f,g^\prime}_\fH-\braket{f^\prime,g}_\fH)
\]
for $\whf=(f,f^\prime)\in\fH^2$ and
$\whg=(g,g^\prime)\in\fH^2$, provided that the
$\fH$-scalar product $\braket{\cdot,\cdot}_\fH$ is
conjugate-linear in the first argument. The same
applies to the $J_\cH$-metric $[\cdot,\cdot]_{\cH^2}$.

The Krein space adjoint $\Gamma^{[*]}$ of $\Gamma$ is defined by
\[
\Gamma^{[*]}:=\{(\whk,\whg)\in\cH^2\times\fH^2\vrt
(\forall(\whf,\whh)\in\Gamma)\,
[\whf,\whg]_{\fH^2}=[\whh,\whk]_{\cH^2}\}\,.
\]
In particular, the inclusion
$\Gamma^{-1}\subseteq\Gamma^{[*]}$ implies that
the Green identity holds:
\begin{equation}
[\whf,\whg]_{\fH^2}=[\whh,\whk]_{\cH^2}\,,\quad
(\whf,\whh)\in\Gamma\,,\quad
(\whg,\whk)\in\Gamma\,.
\label{eq:Green}
\end{equation}

Put $A_*:=\dom\Gamma$ and
\begin{equation}
\Gamma_z:=\Gamma\vrt_{\whfN_z(A_*)}:=
\Gamma\mcap(\whfN_z(A_*)\times\cH^2)\,,\quad
z\in\bbC
\label{eq:Gzz}
\end{equation}
and let $\Gamma^{[*]}_z$ denote the Krein space
adjoint of $\Gamma_z$.
As usual, the eigenspaces of $A_*$ are given by
\[
	\fN_z(A_*):=\ker(A_*-z)\,,\quad
	\whfN_z(A_*):=\{\whf_z=(f_z,zf_z)\vrt
	f_z\in\fN_z(A_*)\}
\]
and similarly for other linear relations.

If $\Gamma$ is isometric, then
$\Gamma_z\subseteq\Gamma$ implies that
\begin{equation}
\Gamma^{-1}_w\subseteq\Gamma^{-1}\subseteq
\Gamma^{[*]}\subseteq\Gamma^{[*]}_z
\label{eq:xx}
\end{equation}
for $w\in\bbC$; that is, $\Gamma_z$ is also isometric.

The $J_\fH$-orthogonal complement
(recall \eg \cite[Definition~1.11]{Azizov89})
$\whfN_z(A_*)^{[\bot]}$ of $\whfN_z(A_*)=\dom\Gamma_z$
is written in terms of $\Gamma_z$ thus
\[
\whfN_z(A_*)^{[\bot]}=(\dom\Gamma_z)^{[\bot]}
=\mul\Gamma^{[*]}_z\,.
\]
On the other hand, $\whfN_z(A_*)^{[\bot]}$
consists of $(f,f^\prime)\in\fH^2$ such that
$f^\prime-\ol{z}f\in\fN_z(A_*)^{\bot}$
($\bot$ denotes the orthogonal complement in $\fH$); hence
\begin{equation}
\whfN_z(A_*)^{[\bot]}=\ol{z}I_\fH\hsum
(\{0\}\times\fN_z(A_*)^{\bot})
\label{eq:eigendecz-b}
\end{equation}
where $\hsum$ denotes the componentwise sum
\cite[Section~2.4]{Hassi09} and
$I_{\fH}$ (the graph of) the
identity operator in $\fH$.
Because $A_*$ is dense in $A^*=\ol{\dom}\Gamma$, one has that
(see also \cite[Proposition~3.9(i)]{Derkach17})
\begin{equation}
\fN_z(A_*)^{\bot}=\fN_z(A^*)^{\bot}=\fM_{\ol{z}}:=
\ol{\ran}(A-\ol{z})\,.
\label{eq:AAA}
\end{equation}
\begin{lem}\label{lem:main}
Let $\Gamma$ be an isometric linear relation
from a $J_\fH$-space to a $J_\cH$-space,
and define
\[
A:=\mul\Gamma^{[*]}\,,\quad
A_\star:=\dom\ol{\Gamma}\,,\quad
A_\#:=\dom\Gamma_\#\,,\quad
\Gamma_\#:=(\Gamma^{[*]})^{-1}\,,
\]
\[
M_{\Gamma_\#}(z):=\Gamma_\#(\whfN_z(A_\#))\,,
\quad z\in\bbC\,.
\]
The following statements are equivalent:
\begin{itemize}
\item[$\mrm{(i)}$]
$\ker\Gamma^{[*]}_z=M_{\Gamma_\#}(\ol{z})$.
\item[$\mrm{(ii)}$]
$M_{\Gamma_\#}(\ol{z})$ is a closed linear relation
in $\cH$.
\item[$\mrm{(iii)}$]
$A^*=A_\star\hsum\whfN_z(A^*)$.
\end{itemize}
\end{lem}
\begin{proof}
(i) $\Rightarrow$ (ii) is clear, since
the kernel of a closed linear relation
$\Gamma^{[*]}_z$ is closed.

(ii) $\Rightarrow$ (i)
By using \eqref{eq:Gzz} we have
\[
\Gamma^{[*]}_z=
\bigl(\Gamma\mcap(\whfN_z(A_*)\times\cH^2)\bigr)^{[*]}
=\ol{\Gamma^{[*]}\hsum\fL_z}\,,\quad
\fL_z:=\{0\}\times \whfN_z(A_*)^{[\bot]}\,.
\]
Then
\[
\ker\Gamma^{[*]}_z=\ker\ol{\Gamma^{[*]}\hsum\fL_z}
=\ol{\ker}(\Gamma^{[*]}\hsum\fL_z)\,.
\]
From here we see that
\[
\ker\Gamma^{[*]}_z=\ker(\Gamma^{[*]}\hsum\fL_z)
\quad\Leftrightarrow\quad
\ker(\Gamma^{[*]}\hsum\fL_z)=
\ol{\ker}(\Gamma^{[*]}\hsum\fL_z)\,.
\]
By using \eqref{eq:eigendecz-b} and \eqref{eq:AAA}
\[
\ker(\Gamma^{[*]}\hsum\fL_z)=\Gamma_\#
\bigl(\whfN_{\ol{z}}(A_\#)
\hsum\Delta_{\ol{z}}
\bigr)\,,\quad
\whfN_{\ol{z}}(A_\#)=A_\#\mcap
\ol{z}I_\fH\,,
\]
\[
\Delta_{\ol{z}}:=A_\#\mcap
(\{0\}\times\fM_{\ol{z}})
=\{0\}\times(\mul A_\#\mcap\fM_{\ol{z}})\,.
\]
But $A\subseteq A_\#$ implies that
\[
\mul A_\#\mcap\fM_{\ol{z}}=\mul A\quad\text{so that}\quad
\Delta_{\ol{z}}=\{0\}\times\mul A
\subseteq\ker\Gamma_\#=A
\]
and it therefore follows that
\[
\ker(\Gamma^{[*]}\hsum\fL_z)=
\Gamma_\#(\whfN_{\ol{z}}(A_\#))=:M_{\Gamma_\#}(\ol{z})\,.
\]

(i) $\Leftrightarrow$ (iii)
By arguing as in \cite[Lemma~2.10]{Hassi09},
the componentwise sum of two closed linear relations
($\Gamma^{[*]}$ and $\fL_z$)
is closed iff the componentwise sum of their adjoints
(and hence of their Krein space adjoints) is a closed
linear relation, \ie
\[
\ol{\Gamma^{[*]}\hsum\fL_z}=
\Gamma^{[*]}\hsum\fL_z
\]
iff
\[
\ol{\Gamma}\hsum\bigl(\whfN_z(A^*)\times\cH^2\bigr)=
\bigl(A_\star\hsum\whfN_z(A^*) \bigr)\times\cH^2
\]
is closed, where we also use \eqref{eq:AAA}.
Because $\ol{A_\star}=A^*$ and
\[
\ol{A_\star\hsum\whfN_z(A^*)}=
\ol{A^*\hsum\whfN_z(A^*)}=A^*
\]
we get that $\Gamma^{[*]}\hsum\fL_z$ is closed iff
\[
A_\star\hsum\whfN_z(A^*)=
\ol{A_\star\hsum\whfN_z(A^*)}=A^*\,.
\]
By the proof of (i) it therefore follows that
\[
\ker\Gamma^{[*]}_z=M_{\Gamma_\#}(\ol{z})
\quad\Leftrightarrow\quad
A^*=A_\star\hsum\whfN_z(A^*)\,.\qedhere
\]
\end{proof}
\section{Weyl family corresponding to an essentially
unitary boundary pair}\label{sec:Weyl}
Here and elsewhere below, a linear relation
$\Gamma$ from a $J_\fH$-space
to a $J_\cH$-space is assumed to be isometric,
unless explicitly stated otherwise. Whenever
we speak of an isometric boundary pair $(\cH,\Gamma)$ for $A^*$,
we assume that $A$ is a closed symmetric linear relation
in $\fH$.

The Weyl family of $A$ corresponding to an isometric
boundary pair $(\cH,\Gamma)$ for $A^*$ is defined by
(see \eg \cite[Definition~3.2]{Derkach17})
\[
M_\Gamma(z):=\Gamma(\whfN_z(A_*))\,,\quad
z\in\bbC_*\,.
\]
In terms of $\Gamma_z$
the linear relation $M_\Gamma(z)$ and its adjoint
in $\cH$ can be described by
\begin{equation}
M_\Gamma(z)=\ran\Gamma_z\,,\quad
M_\Gamma(z)^*=(\ran\Gamma_z)^{[\bot]}=\ker\Gamma^{[*]}_z
\label{eq:MM*}
\end{equation}
It follows from the Green identity \eqref{eq:Green}
and \eqref{eq:MM*} that
$M_\Gamma(\ol{z})\subseteq M_\Gamma(z)^*$.
By \eqref{eq:MM*}, the intersection
$M_\Gamma(z)\mcap M_\Gamma(z)^*$
is a subset of the set of neutral vectors
\cite[Definition~1.3]{Azizov89} of a $J_\cH$-space.
Thus, by applying \eqref{eq:Green}
to $(\whf_z,\whh)\in\Gamma_z$ such that
$[\whh,\whh]_{\cH^2}=0$, one finds that
$M_\Gamma(z)\mcap M_\Gamma(z)^*=\mul\Gamma_z=\mul\Gamma$.
This result is stated without proof in
\cite[Lemma~3.6(i)]{Derkach17}, \cite[Lemma~7.52(i)]{Derkach12},
and is shown in \cite[Lemma~4.1(i)]{Derkach06}
for a unitary boundary pair $(\cH,\Gamma)$. By using
Lemma~\ref{lem:main}, one can find other
invariance results for $M_\Gamma$, which we
do not repeat here.

The following corollary is an application
of Lemma~\ref{lem:main} and \eqref{eq:MM*}.
\begin{cor}\label{cor:0}
Let $(\cH,\Gamma)$ be an essentially unitary boundary pair
for $A^*$, with the Weyl family $M_\Gamma$, and let
$M_{\ol{\Gamma}}$ be the Weyl family corresponding
to a unitary boundary pair $(\cH,\ol{\Gamma})$.
Then $M_\Gamma(z)^*=M_{\ol{\Gamma}}(\ol{z})$,
$z\in\bbC_*$.
\end{cor}
\begin{proof}
Put $\ol{\Gamma}=\Gamma_\#$ in the lemma;
then $A_\star=A_\#$. Since
$M_{\Gamma_\#}(\ol{z})=M_{\ol{\Gamma}}(\ol{z})$
and $M_{\ol{\Gamma}}(\ol{z})$ is closed
by \cite[Theorem~3.9]{Derkach06}, it follows that
$M_\Gamma(z)^*=M_{\ol{\Gamma}}(\ol{z})$.
\end{proof}
Clearly if $\Gamma$ is unitary,
we get that $M_\Gamma(z)^*=M_\Gamma(\ol{z})$
for $z\in\bbC_*$.
\begin{rem}
In the proof of Corollary~\ref{cor:0} we use the
fact that the Weyl family $M_{\ol{\Gamma}}$
corresponding to a unitary boundary pair
$(\cH,\ol{\Gamma})$ for $A^*$ is a Nevanlinna
family, and hence is in particular defined by a closed
linear relation $M_{\ol{\Gamma}}(z)$, $z\in\bbC_*$.
If we did not rely on the present fact, we would instead
obtain from (iii) in Lemma~\ref{lem:main} a weaker
variant of Corollary~\ref{cor:0}:
$M_\Gamma(z)^*=M_{\ol{\Gamma}}(\ol{z})$ for an essentially
unitary boundary pair $(\cH,\Gamma)$ such that
$(\cH,\ol{\Gamma})$ is an $S$-generalized
boundary pair $(\cH,\ol{\Gamma})$, \ie such that
$A_0:=\ker(\ol{\Gamma})_0$ is additionally assumed to be
self-adjoint (\cite[Definition~5.11]{Derkach17}).
Here the linear relation
$(\ol{\Gamma})_0:=\{(\whf,h)\vrt(\exists h^\prime)\,
(\whf,(h,h^\prime))\in\ol{\Gamma} \}$. The argumentation
is due to von Neumann formula
$A^*=A_0\hsum\whfN_z(A^*)$, $z\in\bbC_*$, from which
one has
$A_\star=A_0\hsum\whfN_z(A_\star)$, since
$A_0\subseteq A_\star$; for the latter representation
of $A_\star$, statement (iii) is true
(see also \cite[Theorem~5.17]{Derkach17}).
Vice verse, because by Corollary~\ref{cor:0}
$M_{\Gamma_\#}(\ol{z})=M_{\ol{\Gamma}}(\ol{z})$
is closed, we conclude from Lemma~\ref{lem:main}
that (iii) must be true for $\Gamma$ essentially unitary.
\end{rem}
\begin{cor}
Let $(\cH,\Gamma)$ be an essentially
unitary boundary pair for $A^*$, and let
$A_\star:=\dom\ol{\Gamma}$. Then
$A^*=A_\star\hsum\whfN_z(A^*)$ for all $z\in\bbC_*$.
\end{cor}
\section{Nevanlinna families}\label{sec:Nevaln}
The following definition of a Nevanlinna family is due to
\cite[Definition~9.12]{Derkach12},
\cite[Definition~2.1]{Behrndt08},
\cite[Section~2.6]{Derkach06}.
\begin{defn}\label{defn:Nevaln}
A family $M(z)$, $z\in\bbC_*$, of linear relations
in $\cH$ belongs to the class $\wtcR(\cH)$ of
Nevanlinna families, or is said to be a Nevanlinna family, if:
\begin{itemize}
\item[$(a)$]
For $\Im z>0$ ($\Im z<0$), the relation $M(z)$ is maximal
dissipative (accumulative),
and the operator family $(M(z)+w)^{-1}\in\mrm{B}(\cH)$,
$w\in\bbC_+$ ($\bbC_-$), is analytic;
\item[$(b)$]
$M(z)^*=M(\ol{z})$.
\end{itemize}
\end{defn}
Moreover:
$M\in\cR[\cH]$ if $M\in\wtcR(\cH)$ and
$\dom M(z)=\cH$;
$M\in\cR^s[\cH]$ if $M\in\cR[\cH]$ and $\ker\Im M(z)=\{0\}$;
$M\in\cR^u[\cH]$ if $M\in\cR^s[\cH]$ and $0\in\res\Im M(z)$
($\res$ labels the resolvent set). These subclasses
are all covered by the subclass $\cR(\cH)\subset\wtcR(\cH)$
of Nevanlinna families $M$ such that $M(z)$ is an operator.
For more on the
classification of Nevanlinna families the reader
may refer to \cite{Derkach06}.

Recall that $\bbC_+$ ($\bbC_-$) is the set of $z\in\bbC$
such that $\Im z>0$ ($\Im z<0$). A linear relation $M(z)$
is dissipative (resp. accumulative)
if $(\forall(h,h^\prime)\in M(z))$
$\Im\braket{h,h^\prime}_\cH\geq0$ (resp. $\leq0$);
we emphasize that
the $\cH$-scalar product is conjugate-linear in the first
argument.
A dissipative (resp. accumulative)
$M(z)$ is maximal dissipative
(resp. maximal accumulative) if $M(z)$ has no proper
dissipative (resp. accumulative) extensions.

The Weyl family $M_\Gamma(z)$, $z\in\bbC_*$,
corresponding to an isometric boundary pair $(\cH,\Gamma)$
for $A^*\supseteq A=\ol{A}$
is dissipative (accumulative) for
$\Im z>0$ ($\Im z<0$). Indeed, in view of \eqref{eq:MM*},
$\whh=(h,h^\prime)\in M_\Gamma(z)$ implies that
$(\whf_z,\whh)\in\Gamma_z$ for some $\whf_z\in\whfN_z(A_*)$.
Then, by the Green identity \eqref{eq:Green},
$\Im\braket{h,h^\prime}_\cH=(\Im z)\norm{f_z}^2_\fH$; hence
the claim. But then $(M_\Gamma(z)+w)^{-1}$
is an operator family by
\cite[Theorem~3.1(i)]{Dijksma74}.

If in addition $M_\Gamma(z)^*=M_\Gamma(\ol{z})$, then
$M_\Gamma(\ol{z})^*=M_\Gamma(z)$, and therefore
each member of the Weyl family is closed in this case:
$M_\Gamma(z)^{**}=M_\Gamma(\ol{z})^*=M_\Gamma(z)$.
But then the operator $(M_\Gamma(z)+w)^{-1}\in\mrm{B}(\cH)$
by \cite[Theorem~3.1(vi)]{Dijksma74}, and the relation
$M_\Gamma(z)$ is maximal dissipative (accumulative)
by \cite[Theorem~3.4(ii)]{Dijksma74}.

By the above we conclude the following:
\begin{lem}\label{prop:3}
Let $M_\Gamma$ be the Weyl family corresponding to
an isometric boundary pair $(\cH,\Gamma)$ for $A^*$.
If $M_\Gamma(z)^*=M_\Gamma(\ol{z})$, $z\in\bbC_*$,
then $M_\Gamma$ is a Nevanlinna family .\qed
\end{lem}
By applying Corollary~\ref{cor:0} and Lemma~\ref{prop:3}
we deduce Theorem~\ref{thm:main}.
\begin{rem}
Recall that the Weyl family of $A$ and of its
simple part coincide. Indeed,
let $A_s$ be the simple part
\cite[Proposition~1.1]{Langer77} of $A$ and let
$\Gamma_s$ be the restriction to $\fH^2_s$ of $\Gamma$,
where $\fH_s$ is the closed linear span of
$\{\fN_z(A^*)\vrt z\in\bbC_*\}$.
Put $A_{s*}:=\dom\Gamma_s=A_*\mcap\fH^2_s$.
Then $\whfN_z(A_{s*})=\whfN_z(A_*)\mcap\fH^2_s=\whfN_z(A_*)$.
Thus, since $\Gamma$ is isometric,
$\Gamma_s\subseteq\Gamma$ is also isometric,
and the corresponding Weyl family
of $A_s$ is given by $M_{\Gamma_s}(z)=M_\Gamma(z)$, $z\in\bbC_*$,
by noting that $\Gamma_s\mcap\Gamma_z=\Gamma_z$.
In addition, given an isometric $\Gamma$,
assume that $\Gamma_s$ is essentially unitary.
Then $\Gamma$ is also essentially unitary,
whose closure $\ol{\Gamma}=\ol{\Gamma_s}$.
\end{rem}
\section{Example of an essentially unitary
boundary triple}\label{sec:exam}
An isometric boundary
pair $(\cH,\Gamma)$ for $A^*\supseteq A=\ol{A}$
having the property $\mul\Gamma=\{0\}$
is named by an isometric boundary triple
$(\cH,\Gamma_0,\Gamma_1)$
(\cite[Section~3.1]{Derkach17}). Then
a linear relation $\Gamma$ is identified
with an operator $\whf\mapsto(\Gamma_0\whf,\Gamma_1\whf)$
from $A_*:=\dom\Gamma$ to $\cH^2$, and the corresponding
Weyl function $M_\Gamma$ satisfies
$M_\Gamma(z)\Gamma_0=\Gamma_1$ on $\whfN_z(A_*)$,
$z\in\bbC_*$. When $\Gamma$ is in addition essentially unitary,
the triple $(\cH,\Gamma_0,\Gamma_1)$ is called
an essentially unitary boundary triple.
We recall from \cite[Lemma~4.1(ii)]{Derkach06}
that $\mul\Gamma=\{0\}$ implies $\mul M_\Gamma(z)=\{0\}$,
\ie $M_\Gamma\in\cR(\cH)$, but note that
$\dom M_\Gamma(z)=\Gamma_0(\whfN_z(A_*))$
is in general a proper subset of $\cH$.
In case the triple $(\cH,\Gamma_0,\Gamma_1)$ is such that
$A_*=A^*$ and $\Gamma$ is unitary, one has
$\ran\Gamma=\cH^2$ by \cite[Corollary~2.4]{Derkach06},
and the triple is called an ordinary boundary triple
for $A^*$. In case $A$ is densely defined (\ie $A^*$
is an operator), $A^*$ is identified with its graph.

For the illustration of Corollary~\ref{cor:main},
we consider an example of an essentially unitary boundary triple
(\cf \cite{Jursenas18}).
Let $A$ be a densely defined, closed, symmetric operator
in a Hilbert space $\fH$ with defect numbers $(d,d)$.
Let $L$ be a self-adjoint extension
of $A$ in $\fH$. For simplicity, we assume that $L$
is lower semibounded.
By the von Neumann formula
the adjoint $A^*\supseteq L$ is described by
$\dom A^*=\dom L\dsum\fN_z(A^*)$, where the eigenspace
is spanned by the deficiency elements
$g_\sigma(z)$, $z\in\res L$, with $\sigma$ ranging
over an index set $\cS$ of cardinality $d$. That is,
$\fN_z(A^*)=g_z(\bbC^d)$, where
$g_z(c):=\sum_{\sigma\in\cS}c_\sigma g_\sigma(z)$
for $c=(c_\sigma)\in\bbC^d$.

Let $\fH_{n+1}\subset\fH_n$, $n\in\bbZ$,
be the scale of Hilbert spaces associated with $L$
(see \eg \cite{Albeverio00});
in particular, $\fH_2=\dom L$ and
$\fH_0=\fH$. Then a deficiency element
$g_\sigma(z)\in\fH_0\setm\fH_1$ can be defined
in the generalized sense as $g_\sigma(z)=(L-z)^{-1}\vp_\sigma$
for the functional $\vp_\sigma\in\fH_{-2}\setm\fH_{-1}$.
Thus, $A$ is the symmetric restriction of $L$ to the
domain of $u\in\fH_2$ such that $\braket{\vp,u}=0$;
we use the vector notation
$\braket{\vp,\cdot}=(\braket{\vp_\sigma,\cdot})\co
\fH_2\lto\bbC^d$.
\subsection{Finite rank perturbation of \texorpdfstring{$L$}{}}
Fix $m\in\bbN$ and consider the set
\[
\fK:=\spn\{g_\alpha:=g_\sigma(z_j)\vrt
\alpha=(\sigma,j)\in\cS\times J\}\,,\quad
J:=\{1,2,\ldots,m\}\,.
\]
The points $z_j\in\res L\mcap\bbR$ are such that
$z_j\neq z_{j^\prime}$
for $j\neq j^\prime$ and $j,j^\prime\in J$.
The system $\{g_\alpha\}$
is linearly independent, so the Hermitian (Gram) matrix
\[
	\cG:=(\braket{g_\alpha,g_{\alpha^\prime}}_\fH)
	\in\mrm{B}(\bbC^{md})
\]
is positive definite.

Consider another set
\[
	\fK^\prime:=\fH_{2m+2}\dsum\fM_z\,,\quad
	z\in\res L\setm\cZ\,,\quad\cZ:=\{z_j\vrt j\in J\}
\]
where the subset $\fM_z\subseteq\fH_{2m}\subseteq\fH_2$
is defined by
\[
	\fM_z:=\spn\{\sum_{j\in J}b_j(z_j)^{-1}
	(L-z_j)^{-1}g_\sigma(z)\vrt\sigma\in\cS\}
\]
where the multiplier
\[
	b_j(z_j):=\prod_{j^\prime\in J\setm\{j\}}(z_j-z_{j^\prime})
\]
for $m>1$, and $b_1(z):=1$ for $m=1$ and $z\in\bbC$.
Define the operator $K$ in $\fH$ by
\[
\dom K:=\fK^\prime\dsum\fK\,,\quad
K(u+k):=Lu+\sum_{\alpha,\alpha^\prime\in\cS\times J}
C_{\alpha\alpha^\prime}
\braket{g_{\alpha^\prime},k}_\fH g_\alpha
\]
for $u+k\in\fK^\prime\dsum\fK$,
where $C_{\sigma j,\sigma^\prime j^\prime}:=
z_j[\cG^{-1}]_{\sigma j,\sigma^\prime j^\prime}$.
Then $K$ is referred to as the rank-$md$ perturbation of $L$.
Note that the sum defining $\dom K$ is direct.
\subsection{Boundary value space}
Let $\cH:=\bbC^d$ and
define the operator $\Gamma$ (identified with its graph) by
\[
\Gamma:=\{\bigl(
(u+k,K(u+k)),(c(k),\braket{\vp,u}+\cM d(k))\bigr)\vrt
u+k\in\fK^\prime\dsum\fK\}\,.
\]
Here the column-vectors
\begin{align*}
c(k)=&(c_\sigma(k))\in\bbC^d\,,\quad
c_\sigma(k):=\sum_{j\in J}d_{\sigma j}(k)\,,
\\
d(k)=&(d_\alpha(k))\in\bbC^{md}\,,\quad
d_\alpha(k):=\sum_{\alpha^\prime\in\cS\times J}
[\cG^{-1}]_{\alpha\alpha^\prime}
\braket{g_{\alpha^\prime},k}_\fH
\end{align*}
and the matrix
\[
	\cM=(\cM_{\sigma \alpha^\prime})
	\in\mrm{B}(\bbC^{md},\bbC^d)\,,\quad
	\cM_{\sigma,\sigma^\prime j^\prime}:=
	R_{\sigma\sigma^\prime}(z_{j^\prime})
\]
for some matrix-valued Nevanlinna family
$R(\cdot)=(R_{\sigma\sigma^\prime}(\cdot))$
of class $\cR^s[\cH]$. In fact, if the functional
$\braket{\vp^{\mrm{ex}},\cdot}=
(\braket{\vp^{\mrm{ex}}_\sigma,\cdot})$ extends
$\braket{\vp,\cdot}\co\fH_2\lto\cH$ to $\dom A^*$
according to (see also \cite[Section~3.1.3]{Albeverio00},
\cite{Hassi09-b})
\[
	\braket{\vp^{\mrm{ex}},u}:=\braket{\vp,u^\#}+R(z)c\,,
	\quad u=u^\#+g_z(c)\,,\quad u^\#\in\fH_2\,,
	\quad c\in\cH
\]
then the matrix $R(z)\in\mrm{B}(\cH)$,
$z\in\res L$, is defined by
\[
R_{\sigma\sigma^\prime}(z):=
\braket{\vp^{\mrm{ex}}_\sigma,g_{\sigma^\prime}(z)}\,.
\]
From here one verifies that
$\ker\Im R(z)=\{0\}$:
By definition
the imaginary part $\Im R(z)$ is the matrix with entries
$(\Im z)\braket{g_\sigma(z),g_{\sigma^\prime}(z)}_\fH$.
Then the kernel $\ker\Im R(z)$ is the set of $c\in\cH$ such that
$g_z(c)\in g_z(\cH)^\bot$; hence $c=0$.

An ordinary boundary triple
$(\bbC^d,\wtGm_0,\wtGm_1)$ for $A^*$
is defined by
\[
	\wtGm_0(u^\#+g_z(c)):=c\,,\quad
	\wtGm_1(u^\#+g_z(c)):=
	\braket{\vp,u^\#}+R(z)c
\]
with $u^\#\in\fH_2$ and $c\in\cH$, and with
$R(z)$ as described above, while
$\Gamma$ defines the boundary value space of operator $K$. Indeed,
associate with $\Gamma$ two single-valued linear relations
\begin{align*}
\Gamma_0:=&\{(\whf,h)\vrt(\exists h^\prime\in\cH)\,
(\whf,\whh)\in\Gamma\,;\whh=(h,h^\prime)\}\,,
\\
\Gamma_1:=&\{(\whf,h^\prime)\vrt(\exists h\in\cH)\,
(\whf,\whh)\in\Gamma\,;\whh=(h,h^\prime)\}\,.
\end{align*}
Then the boundary form of the operator $K$
is given by
\[
	\braket{u,Kv}_\fH-\braket{Ku,v}_\fH=
	\braket{\Gamma_0u,\Gamma_1v}_{\cH}-
	\braket{\Gamma_1u,\Gamma_0v}_{\cH}
\]
for $u,v\in\dom K$, provided that $\Gamma_0$ (resp. $\Gamma_1$)
is regarded as the mapping $\dom K\lto\cH$.
\subsection{Linear relation \texorpdfstring{$\Gamma_z$}{} and
its Krein space adjoint}
Let $K^*$ be the adjoint of $K$ in $\fH$.
\begin{prop}\label{prop:KA}
$K^*=A$.
\end{prop}
\begin{proof}
The (graph of the) adjoint $K^*$ consists of
$(y,x)\in\fH\times\fH$ such that
$(\forall u+k\in\fK^\prime\dsum\fK)$
\[
	\braket{u+k,x}_\fH=\braket{K(u+k),y}_\fH\,.
\]
Let $P$ be an orthogonal projection in $\fH$ onto $\fK$.
Then $x=x_\bot+k_x\in\fK^\bot\op\fK$ and
$y=y_\bot+k_y\in\fK^\bot\op\fK$, where
$\fK^\bot:=\fH\om\fK$. Since $\fK^\prime\subseteq\fK^\bot$
it follows that
\[
	\braket{u+k,x}_\fH=\braket{u,x}_\fH+
	\braket{d(k),\cG d(k_x)}_{\bbC^{md}}
\]
and
\[
	\braket{K(u+k),y}_\fH=
	\braket{Lu,y}_\fH+
	\braket{d(k),Z_d\cG d(k_y)}_{\bbC^{md}}
\]
where $Z_d$ is the matrix direct sum of $d$ diagonal
matrices $\diag\{z_j;j\in J\}$. Because the set
$\fK^\prime\supseteq\fH_{2m+2}$, \ie is a dense
subset of $\fH_2$, it follows from the above that
$y\in\fH_2$ and $x=Ly$ and $d(k_x)=\cG^{-1}Z_d\cG d(k_y)$.
On the other hand, $k_x=PLy$ implies that
\[
d(k_x)=\cX\braket{\vp,y}+\cG^{-1}Z_d\cG d(k_y)\,,
\quad
\cX=(\cX_{\alpha\sigma^\prime})
\in\mrm{B}(\bbC^d,\bbC^{md})\,,\quad
\cX_{\alpha\sigma^\prime}:=\sum_{j^\prime\in J}
[\cG^{-1}]_{\alpha,\sigma^\prime j^\prime}\,.
\]
Since $\cX$ is invertible, one deduces that
$\braket{\vp,y}=0$, and hence $y\in\dom A$.
\end{proof}
Let $A_*:=\dom\Gamma=K$. By Proposition~\ref{prop:KA}
$A_*$ is dense in $A^*$, \ie $\ol{A_*}=A^*$.
One also verifies that the eigenspace
$\fN_z(A_*)=\fN_z(A^*)$ for $z\in\res L\setm\cZ$.
Then, the single-valued linear relation $\Gamma_z$
(recall \eqref{eq:Gzz})
and its Krein space adjoint $\Gamma^{[*]}_z$ are given by
\begin{align*}
\Gamma_z=&\{\bigl((g_z(c),zg_z(c)),(c,R(z)c)\bigr)\vrt
c\in\cH\}\,,
\\
\Gamma^{[*]}_z=&\{\bigl((c,\braket{\vp,u}+R(\ol{z})c),
(g,Lu+\ol{z}(g-u))\bigr)
\vrt c\in\cH\,;\,u\in\fH_2\,;\,g\in\fH\}\,.
\end{align*}
It follows that $\Gamma_z\subseteq\Gamma$ and
$\Gamma^{-1}_w\subseteq\Gamma^{[*]}_z$ for all
$z,w\in\res L\setm\cZ$. Moreover
\[
\ker\Gamma^{[*]}_z=R(\ol{z})\,,\quad
z\in\res L\setm\cZ\,.
\]
\subsection{Weyl function}
Let $\Gamma^{[*]}$ be the Krein space adjoint of $\Gamma$.
\begin{prop}
$\Gamma^{[*]}=\ol{\Gamma}^{\;-1}$, with the closure
$\ol{\Gamma}=(\wtGm_0,\wtGm_1)$.
\end{prop}
\begin{proof}
Step 1.
By definition,
$\Gamma^{[*]}$ consists of
$((c,c^\prime),(u,u^\prime))\in\cH^2\times\fH^2$ such that
\[
u^\prime=L_{-2}u-\omega c\,,\quad
\omega c:=\sum_{\sigma\in\cS}c_\sigma\vp_\sigma\,,\quad
c=(c_\sigma)\in\cH
\]
and
\[
d(k_{u^\prime})=\cG^{-1}[Z_d\cG d(k_u)-\cM^*c]
+\cX c^\prime\,,\quad
k_u:=Pu\in\fK\,,\quad
k_{u^\prime}:=Pu^\prime\in\fK
\]
where $L_{-2}$ is a bounded continuation to $\fH_{-2}$ of $L$,
and where an orthogonal projection $P$ and the matrices
$Z_d$ and $\cX$ are as in the proof of
Proposition~\ref{prop:KA}. Using
$u=u_\bot+k_u\in\fK^\bot\op\fK$ and
\[
L_{-2}u=L_{-2}u_\bot+\sum_{\alpha\in\cS\times J}
[Z_dd(k_u)]_\alpha g_\alpha+\omega c(k_u)
\]
one concludes that $u\in\dom\tau$ and
$u^\prime=\tau u$ and $c=c(k_u)=:\Gamma^\tau_0 u$,
where the operator
\[
\tau:=\{(u,Lu_\bot+\sum_{\alpha\in\cS\times J}
[Z_dd(k_u)]_\alpha g_\alpha)\vrt u_\bot\in\fH_2
\mcap\fK^\bot\}\,.
\]
Using in addition that
\[
\cG^{-1}[Z_d\cG d(k_u)-\cM^*c(k_u)]=
(Z_d-\cX\cM)d(k_u)
\]
one finds that
\[
d(P\tau u)=Z_dd(k_u)+\cX(c^\prime-\cM d(k_u))\,.
\]
On the other hand, by the definition of $\tau$
\[
d(P\tau u)=Z_dd(k_u)+\cX
\braket{\vp,u_\bot}\,.
\]
Thus
\[
c^\prime=\braket{\vp,u_\bot}+\cM d(k_u)=:
\Gamma^\tau_1u
\]
and subsequently
\[
\Gamma^{[*]}=\{((\Gamma^\tau_0 u,\Gamma^\tau_1 u),
(u,\tau u))\vrt u\in\dom\tau\}\,.
\]

Step 2.
The boundary form of $\tau$ satisfies
an abstract Green identity
\[
	\braket{u,\tau v}_\fH-\braket{\tau u,v}_\fH=
	\braket{\Gamma^\tau_0 u,\Gamma^\tau_1 v}_\cH-
	\braket{\Gamma^\tau_1 u,\Gamma^\tau_0 v}_\cH
\]
for $u,v\in\dom\tau$, and
for $\Gamma^\tau_0$ and $\Gamma^\tau_1$ as defined in step 1.
Therefore, to show that
\[
\Gamma^{[*]}=\{((\wtGm_0 u,\wtGm_1 u),
(u,A^* u))\vrt u\in\dom A^*\}
\]
it remains to prove that $\tau=A^*$ and
$\Gamma^\tau_i=\wtGm_i$, $i\in\{0,1\}$.

By direct computation, and using the equivalence
relation $\fH_2\mcap\fK\ni k\Leftrightarrow c(k)=0$,
one finds that the adjoint $\tau^*=A$; since $\tau$ and
$A$ are closed, this yields $\tau=A^*$.

Consider $u=u_\bot+k_u\in\dom\tau=\dom A^*$; then
$u=u^\#+g_z(c)$ with $c:=c(k_u)$ and
$u^\#:=u_\bot+k_u-g_z(c)$. Since $u_\bot\in\fH_2$
and $k_u-g_z(c)\in\fH_2$, it holds $u^\#\in\fH_2$. Then
\[
\Gamma^\tau_0u=c=\wtGm_0u
\]
and
\begin{align*}
\Gamma^\tau_1u=&\braket{\vp,u^\#}+
\braket{\vp,g_z(c)-k_u}+\cM d(k_u)
\\
=&\braket{\vp,u^\#}+R(z)c=\wtGm_1u\,.
\end{align*}
This completes the proof of
$(\Gamma^{[*]})^{-1}=(\wtGm_0,\wtGm_1)$.

Step 3.
The closure $\ol{\Gamma}$ is the adjoint of
$J_\fH\Gamma^{[*]}J_\cH$; hence it consists of
$((y,y^\prime),(x,x^\prime))\in\fH^2\times\cH^2$ such that
\[
\braket{u,y^\prime}_\fH-\braket{A^*u,y}_\fH=
\braket{\wtGm_0 u,x^\prime}_{\cH}-
\braket{\wtGm_1 u,x}_{\cH}
\]
for all $u\in\dom A^*$;
\ie $\ol{\Gamma}=(\Gamma^{[*]})^{-1}$.
\end{proof}
Observe that
$\Gamma^{-1}\subseteq\Gamma^{[*]}\subseteq\Gamma^{[*]}_w$
for all $w\in\res L\setm\cZ$, so \eqref{eq:xx} holds true.
Note also that $A=\ker\ol{\Gamma}=\ker\Gamma
\subseteq A_*=K\subseteq A^*=\dom\ol{\Gamma}$.

Since $\Gamma$ is essentially unitary and $\ol{A_*}=A^*$,
the triple $(\cH,\Gamma_0,\Gamma_1)$
is an essentially unitary boundary triple for $A^*$.
The triple $(\cH,\wtGm_0,\wtGm_1)$ is an ordinary boundary
triple for $A^*$, with the associated Weyl family defined by
$M_{\ol{\Gamma}}(z)=R(z)$, $z\in\res L$.

The domain
\[
\dom M_\Gamma(z)=\Gamma_0(\whfN_z(A_*))=\cH
\]
so by Corollary~\ref{cor:main} the Weyl function
\[
M_\Gamma(z)=M_{\ol{\Gamma}}(z)=R(z)\,,\quad
z\in\res L
\]
and therefore $M_\Gamma$
belongs to a Nevanlinna family of class $\cR^s[\cH]$.
That $M_\Gamma=R$ can be also checked
by computing $\Gamma(\whfN_z(A_*))$ directly.
Moreover, since $\ran\ol{\Gamma}=\cH^2$, one concludes
that actually $R\in\cR^u[\cH]$.

\newcommand{\etalchar}[1]{$^{#1}$}

\end{document}